\documentclass{amsart}
\usepackage{amssymb}
\usepackage{braket}
\usepackage{mathrsfs}
\usepackage{amsrefs}
\usepackage{ifthen}
\usepackage{tikz}

\theoremstyle{plain}
\newtheorem{theo}{Theorem}[section]
\newtheorem{prop}[theo]{Proposition}
\newtheorem{lem}[theo]{Lemma}
\newtheorem{cor}[theo]{Corollary}

\newtheorem{condition}{Condition}

\theoremstyle{remark}
\newtheorem{rem}[theo]{Remark}
\newtheorem{ex}[theo]{Example}

\newtheorem*{acknowledgement}{Acknowledgement}

\newcommand{\Z}{\mathbb{Z}}
\newcommand{\R}{\mathbb{R}}
\newcommand{\C}{\mathbb{C}}
\newcommand{\CP}{\mathbb{CP}}
\newcommand{\Area}{\mathrm{Area}}
\newcommand{\Hyp}{\mathrm{Hyp}}
\newcommand{\Half}{\mathrm{Half}}
\newcommand{\sign}{\mathrm{sign}}
\newcommand{\Harm}{{\mathscr H}}

\newcommand{\Met}{\mathrm{Met}}
\newcommand{\Int}{\mathrm{Int}}
\newcommand{\Si}{\Sigma}

\newcommand{\SW}{Seiberg-Witten }
\newcommand{\config}{{\mathcal C}}
\newcommand{\rconfig}{{\mathcal C}_{\mathbb{R}}}
\newcommand{\cconfig}{{\mathcal C}_{\mathbb{C}}}
\newcommand{\Ker}{\mathop{\mathrm{Ker}}\nolimits}
\newcommand{\im}{\mathop{\mathrm{Im}}\nolimits}
\newcommand{\vnU}{{\mathcal U}}
\newcommand{\id}{\mathrm{id}}
\newcommand{\Hom}{\mathop{\mathrm{Hom}}\nolimits}
\newcommand{\ind}{\mathop{\mathrm{ind}}\nolimits}

\title[Bounds on genus and configurations of embedded surfaces]{Bounds on genus and configurations of embedded surfaces in 4-manifolds}

\author{Hokuto Konno}

\address{Graduate School of Mathematical Sciences, the University of Tokyo, 3-8-1 Komaba, Meguro, Tokyo 153-8914, Japan}

\email{hkonno@ms.u-tokyo.ac.jp}

\date{}

\begin{document}

\maketitle

\begin{abstract}
For several embedded surfaces with zero self-intersection number in $4$-manifolds, we show that an adjunction-type genus bound holds for at least one of the surfaces under certain conditions.
For example, we derive certain adjunction inequalities for surfaces embedded in $m\CP^2\# n(-\CP^2)\ (m, n \geq 2)$.
The proofs of these results are given by studying a family of \SW equations.
\end{abstract}

\tableofcontents

\section{Introduction}

It is a fundamental problem in $4$-dimensional topology to find a lower bound for the genus of an embedded surface which represents a given second homology class of a $4$-dimensional manifold.
In this paper, we consider a configuration consisting of several surfaces embedded in an oriented closed $4$-manifold.
Suppose that the self-intersection numbers of all the surfaces are zero, and the number of the surfaces is more than $b^+$.
The main theorem of this paper is that an adjunction-type genus bound holds for at least one of the surfaces under certain conditions on the surfaces.
For example, although the \SW invariant of $m\CP^2\# n(-\CP^2)\ (m,n \geq 2)$ vanishes for any spin c structure on it,  we can derive the adjunction inequalities for surfaces embedded in this $4$-manifold under certain conditions.
In addition, we also give an alternative proof of the adjunction inequalities by Strle~\cite{MR2064429} for configurations of surfaces with positive self-intersection numbers.
The proofs of these results are given by studying a $b^+$-parameter family of \SW equations obtained by stretching neighborhoods of several embedded surfaces.

The original problem to find a lower bound on genus of a single embedded surface is called the minimal genus problem.
(For the detailed history of the minimal genus problem see Lawson's survey~\cite{MR1486407}.)
Classical results for the minimal genus problem were obtained by using Rochlin's theorem and $G$-signature theorem.
(Kervaire-Milnor~\cite{MR0133134}, Rochlin~\cite{MR0298684}, Hsiang-Szczarba~\cite{MR0339239}.)
After gauge theory appeared in $4$-dimensional topology, it has been a strong tool to study the problem.
Kronheimer-Mrowka~\cite{MR1306022} proved the Thom conjecture, namely, the minimal genus problem for $\CP^2$ by using the \SW equations.
Kronheimer-Mrowka's lower bound on genus is an equality for algebraic curves, where it is referred to as the adjunction formula.
This type of inequality is often called the adjunction inequality.
In Kronheimer-Mrowka's proof, they used in an essential way the fact that the $4$-manifold defined by blow-up for $\CP^2$ has a metric with positive scalar curvature.

On the other hand, Strle~\cite{MR2064429} showed the adjunction inequality for a surface with positive self-intersection number in a $4$-manifold with $b^+=1$ and $b_1=0$ without any assumption for differential geometric structure such as the existence of a metric with positive scalar curvature.
For this reason, the Thom conjecture for a general rational homology $\CP^2$ follows as a special case of Strle's result.
Strle's method is to consider the moduli space of the \SW equations on a $4$-manifold with cylindrical ends.
Recently, Dai-Ho-Li~\cite{MR3465838} derived an alternative simple proof of Strle's theorem in the case of $b^+=1$ and $b_1=0$ and sharper results in the case of $b^+=1$ and $b_1>0$.
Dai-Ho-Li's method is to use the wall crossing formula for the \SW invariants by considering several spin c structures for one surface.

In Strle's paper~\cite{MR2064429}, he also showed the following adjunction inequalities for disjoint embedded surfaces with positive self-intersection numbers.
Let us consider a $4$-manifold with $b^+>1$ and embedded surfaces with positive self-intersection numbers.
Suppose that the surfaces are disjoint and the number of surfaces is $b^+$.
Strle showed that the adjunction-type inequality holds for at least one of these surfaces.
In particular, for $m \geq 2$, Strle's results can be applied to $m\CP^2 = \#_{p=1}^m\CP^2_p$ and its blow-up.
For $m \geq 2$ and $n \geq 0$, note that the \SW invariant of $m\CP^2\# n(-\CP^2)$ vanishes for any spin c structure on it.
Thus it is impossible to show the adjunction inequality on $c$ by proving that the \SW invariant with respect to the spin c structure corresponding to $c$ is non-trivial.
In fact, Nouh~\cite{MR3159966} showed that the adjunction inequality for a single surface in $2\CP^2$ does not hold in general.
In addition, by taking connected sum of Nouh's surfaces and algebraic curves, one can easily construct surfaces in $m\CP^2 \# n(-\CP^2)$ which violate the adjunction inequality.
To our knowledge, the adjunction-type inequalities for $m\CP^2\# n(-\CP^2)$ are only Strle's in previous researches.
(Before Strle's work, Gilmer~\cite{MR603768} have considered configurations of embedded surfaces, although Gilmer's bounds are not adjunction-type.
For a single surface, Gilmer's bound is the one obtained by Rochlin~\cite{MR0298684} and Hsiang-Szczarba~\cite{MR0339239}.)

In this paper, we also consider several embedded surfaces without any assumption for differential geometric structure.
The main theorem of this paper is as follows.
Suppose that the self-intersection numbers of all the given embedded surfaces are zero, and the number of the surfaces is more than $b^+$.
Then an adjunction-type  genus bound holds for at least one of the surfaces under certain conditions on the surfaces.
The conditions are given in terms of intersection numbers with a characteristic in the adjunction inequalities and their mutual geometric intersections.
These conditions are easily described for small $b^+$, so in this paper we will first present the adjunction inequalities for surfaces  in $2\CP^2\# n(-\CP^2)$ as a special case of the main theorem.
For surfaces in $2\CP^2\# n(-\CP^2)$, Strle's bound is the adjunction inequality for at least one of two disjoint surfaces with positive self intersection numbers.
On the other hand,  our genus bound is the adjunction inequality for at least one of four surfaces with self-intersection number zero and we allow certain geometric intersections.
A generalization of this result to $m\CP^2\# n(-\CP^2)$ for any $m \geq 2$ will be also given (Corollary~\ref{cor : adj. ineq for a single surface : mCP^2}).

We will also give an alternative proof  of the adjunction inequalities by Strle for configurations of surfaces with positive self-intersection numbers as a corollary of our main theorem.
In the case of $b^+ = 1$, the setting of our argument is quite similar to that of Dai-Ho-Li: surfaces with zero self-intersection number play a key role as in the original argument due to Kronheimer-Mrowka.
One difference between our method and Dai-Ho-Li's is that we fix a spin c structure and consider two surfaces while Dai-Ho-Li fix one surface and consider two spin c structures.
In the alternative proof of Strle's adjunction inequalities in the case of $b^+ = 1$, our argument can be regarded as a reformulation, without using blow-up formula, of a part of Dai-Ho-Li's argument.

For surfaces with self-intersection number zero, if the surfaces are obtained by the connected sum of surfaces with positive self-intersection numbers and exceptional curves in $-\CP^2$'s, our genus bounds are equivalent to Strle's.
However, for general surfaces, our results are wider generalizations of Strle's, since, of course, surfaces with self-intersection number zero cannot written as such a form of connected sum in general.

The wall crossing phenomena for general $b^+ \geq 1$ play prominent roles in the proof of the main theorem.
 We can grasp the wall crossing phenomena by studying the moduli space of the \SW equations parametrized by a $b^+$-dimensional space in the space of Riemannian metrics.
In other words, we consider a family version of the \SW invariants studied by Li-Liu~\cite{MR1868921}.
We will construct this $b^+$-parameter family of Riemannian metrics by stretching the neighborhoods of surfaces embedded in a suitable configuration.
This construction of the family of Riemannian metrics is a slight generalization of one due to Fr{\o}yshov~\cite{MR2052970} used in the context of instanton Floer homology.
We will also describe positions of metrics in relation to the wall and study a condition to grasp the wall crossing phenomena in terms of several embedded surfaces.
This ``higher-dimensional" wall crossing argument using several embedded surfaces enables us to obtain a solution of the \SW equations with respect to a certain metric even when the \SW invariants of the $4$-manifold vanish.

The outline of this paper is as follows.
In Section~2, we formulate our main theorem and give its consequences.
In Section~\ref{subsection : Special case of the main theorem}, we present a special case of the main theorem (Theorem~\ref{main thm : special}), namely, the adjunction inequalities for configurations of surfaces in $2\CP^2\# n(-\CP^2)$.
In this subsection, we also give the adjunction inequalities for a single surface as a corollary of Theorem~\ref{main thm : special} (Corollary~\ref{cor : adj. ineq for a single surface}).
In Section~\ref{subsection : General form of the main theorem}, we formulate the most general form of our main theorem (Theorem~\ref{general}) and give the adjunction inequalities for a single surface in $m\CP^2\# n(-\CP^2)$ (Corollary~\ref{cor : adj. ineq for a single surface : mCP^2}).
In this subsection, we also give an alternative proof of Strle's adjunction inequalities.
In Section~3, we prove the main theorem assuming analytical Lemma~\ref{stretch} and Proposition~\ref{analysis}.
In Section~4, we give the proof of Lemma~\ref{stretch} and Proposition~\ref{analysis}.

\begin{acknowledgement}
The author would like to express his deep gratitude to Mikio Furuta for the numerous helpful suggestions and continuous encouragements during this work. 
It is also his pleasure to thank Shinichiroh Matsuo for several useful advice and instructing him in a way of making a figure used in this paper.
The author also wishes to thank Tirasan Khandhawit for useful conversations and his comments on the preprint version. 
The author would also like to express his appreciation to Yuichi Yamada for pointing out a mistake in the preprint version. 
The author also thanks Tomohiro Asano for stimulating discussions and also thanks Kouki Sato for answering his questions.
The author also gratefully acknowledge the many helpful suggestions of the anonymous referee.
The author was supported by JSPS KAKENHI Grant Number 16J05569 and
the Program for Leading Graduate Schools, MEXT, Japan.

\end{acknowledgement}

\section{Statement of the main theorem and its consequences}

In this section, we state our main theorem and give its consequences.

\subsection{Special case of the main theorem}
\label{subsection : Special case of the main theorem}

In this subsection, as a special case of our main theorem, we give the adjunction inequalities for embedded surfaces in the $4$-manifold $X = 2\CP^2 \# n(-\CP^2)$.
We will often use the identification $H^2(\cdot) \simeq H_2(\cdot)$ obtained by Poincar\'{e} duality.
In this paper, we consider only surfaces which are oriented, closed and connected.

Let us consider the 4-manifold
\[
X = 2\CP^2 \# n(-\CP^2) = (\#_{p=1}^2 \CP^2_p) \# (\#_{q=1}^n (-\CP^2_{q}))\quad (n > 0).
\]
Let $H_p$ denote a generator of $H_2(\CP^2_p;\Z)$ and $E_q$ a generator of $H_2(-\CP^2_q;\Z)$.
For a cohomology class $c \in H^2(X;\Z)$ and homology classes $\alpha_1, \ldots, \alpha_4 \in H_2(X;\Z)$,
we define a line $L_i\ (i=1,\ldots,4)$ in $\R^2$ by 
\begin{align}
L_i := \Set{ (x_1, x_2) \in \R^2 | (x_1H_1+x_2H_2) \cdot \alpha_i = c \cdot \alpha_i }.
\label{line L_i}
\end{align}
For these lines, we will consider the condition that (parts of) lines $L_1, \ldots, L_4$ form sides of a ``quadrilateral" by this order.
Here we use the word ``quadrilateral" in the following sense.
Let  $L_1', \ldots, L_4'$ be four line segments in $\R^2$.
If an orientation of $L_i'$ is given, we can define the initial point $I(L_i')$ and the terminal point $T(L_i')$ of $L_i'$.
We call the ordered set $(L_1', \ldots, L_4')$ a {\it quadrilateral} when there exists an orientation for each $L_i'$ such that $T(L_i') = I(L_{i+1}')$ holds for each $i \in \Z/4$.
(We admit a point as a line segment.
Thus a ``triangle" is also a quadrilateral in our definition.)

\begin{figure}
\begin{center}
\begin{tikzpicture}
 [scale = 1/2]
\draw(6.7,0) node {$x_1$};
\draw(0,3.6) node {$x_2$};
\draw[->] (-3.3,0) -- (6.3, 0);
\draw[->] (0,-6.3)-- (0,3.3);
\draw (4,-6)-- (1/3,3);
\draw (3,5/4)-- (-3,-1/4);
\draw (6,-14/3)-- (-3,4/3);
\draw (6,-10/3)-- (10/3,-6);
\draw(2.6,-1) node {$L_1$};
\draw(-0.4,0.9) node {$L_2$};
\draw(0.6,-1.5) node {$L_3$};
\draw(5,-5.2) node {$L_4$};
\fill(0,0) circle (4pt);
\label{a}
\end{tikzpicture}
\end{center}
\caption{An example of a quadrilateral including the origin of $\R^2$}
\end{figure}
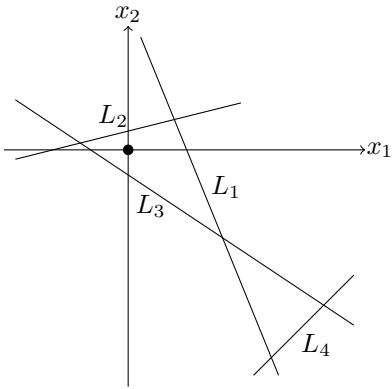

 A special case of our main theorem is as follows.
 Let $\sign(X) = 2-n$ denote the signature of $X$.

\begin{theo}
For the $4$-manifold
\[
X = 2\CP^2 \# n(-\CP^2)\quad (n > 0),
\]
let $c \in H^2(X;\Z)$ be a characteristic with $c^2 > \sign(X)$ and $\alpha_1, \ldots, \alpha_4 \in H_2(X;\Z)$ be homology classes with $\alpha_i^2=0\ (i = 1,\ldots, 4)$.
Let $\Si_1, \ldots, \Si_4 \subset X$ be embedded surfaces with $[\Si_i] = \alpha_i$.
Assume that $\alpha_i$ and $c$ satisfy the following (A) and $\Si_i$ satisfy (B) :
\begin{description}
\item[(A)] The lines $L_1, \ldots, L_4$ form sides of a quadrilateral including the origin of $\R^2$ by this order.
\item[(B)] $\Si_i \cap \Si_{i+1} = \emptyset\ (i \in \Z/4).$
\end{description}
Then, the inequality
\begin{align*}
-\chi(\Si_i) \geq |c \cdot \alpha_i|
\end{align*}
holds for at least one $i \in \{1,\ldots, 4\}$.
Here $\chi(\Si_i)$ is the Euler characteristic of $\Si_i$.
\label{main thm : special}
\end{theo}

Theorem~\ref{main thm : special} is a special case of Theorem~\ref{general} in the next subsection and the proof of Theorem~\ref{general} will be given in Section~\ref{section : Proof of the main theorem}.

\begin{ex}
Let $X = 2\CP^2 \# 19(-\CP^2),\ c = H_1 - 3H_2 -\sum_{q=1}^{19}E_q$.
The homology classes
\begin{align*}
\alpha_1 &:= 3H_1 + 3H_2  - \textstyle{\sum}_{q=1}^3 E_q + \textstyle{\sum}_{q=4}^{10} E_q + 2(E_{11} + E_{12}),\\
\alpha_2 &:= -3H_1 + 2H_2  + \textstyle{\sum}_{q=1}^3 E_q + \textstyle{\sum}_{q=13}^{18} E_q + 2E_{19}\\
\alpha_3 &:= H_1 + H_2 + E_{12} - E_{13},\\
\alpha_4 &:= 2H_1 - H_2 - \textstyle{\sum}_{q=1}^3 E_q - E_{13} + E_{14}
\end{align*}
satisfy that $\alpha_i^2 = 0$ and $\alpha_i \cdot \alpha_{i+1} = 0
\ (i \in \Z/4).$
It is easy to check that these $\alpha_i$ and $c$ satisfy (A) in Theorem~\ref{main thm : special}.
(See Figure~\ref{picture in Ex. main example}.)

Thus, by Theorem~\ref{main thm : special},  for embedded surfaces $\Si_i$ satisfying $[\Si_i] = \alpha_i$, if they also satisfy that $\Si_i \cap \Si_{i+1} = \emptyset\ (i \in \Z/4)$,
\[
-\chi(\Si_i) \geq |c \cdot \alpha_i|
\]
holds for at least one $i \in \{1,\ldots, 4\}$.
This means that the genus bound
\[
g(\Si_i) \geq 2
\]
holds for at least one $i \in \{1,\ldots, 4\}$.

\begin{figure}
\begin{center}
\begin{tikzpicture}
\draw(2.4,0) node {$x_1$};
\draw(0,2.2) node {$x_2$};
\draw[->] (-2.1,0) -- (2.1, 0);
\draw[->] (0,-2.7)-- (0,2);
\draw (-5/6,3/2)-- (3/2,-5/6);
\draw (1/3,3/2)-- (-3/2,-5/4);
\draw (-3/2,-1/2)-- (1/2,-5/2);
\draw (-1/4,-5/2)-- (3/2,1);
\draw(0.55,0.55) node {$L_1$};
\draw(-0.9,0.2) node {$L_2$};
\draw(-0.75,-1.55) node {$L_3$};
\draw(0.6,-1.6) node {$L_4$};
\fill(0,0) circle (2pt);
\end{tikzpicture}
\end{center}
\caption{$L_1, \ldots, L_4$ in Example~\ref{main example}}
\label{picture in Ex. main example}
\end{figure}
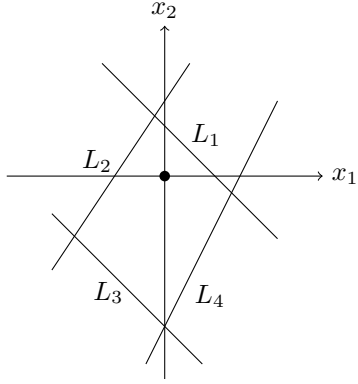

\label{main example}
\end{ex}

\begin{rem}
If $\alpha_1, \ldots, \alpha_4$ and $c$ satisfy (A) in Theorem~\ref{main thm : special}, for any non-zero integer $m$, the homology classes $m\alpha_1, \ldots, m\alpha_4$ and $c$ obviously satisfy (A).
Thus, for $\alpha_1, \ldots, \alpha_4$ in Example~\ref{main example} and embedded surfaces $\Si_i$ satisfying $[\Si_i] = m\alpha_i$ and
$\Si_i \cap \Si_{i+1} = \emptyset\ (i \in \Z/4)$, we can show that
\[
g(\Si_i) \geq |m|+1
\]
holds for at least one $i \in \{1,\ldots, 4\}$ by Theorem~\ref{main thm : special}.
\end{rem}

\begin{rem}
Note that the \SW invariant of $2\CP^2 \# n(-\CP^2)\ (n \geq 0)$ vanishes for any spin c structure.
Thus, for any characteristic $c$ on $2\CP^2 \# n(-\CP^2)$,  it is impossible to derive the adjunction inequality for $c$ by proving that the \SW invariant with respect to the spin c structure corresponding to $c$ is non-trivial.
In fact, for a single surface in $2\CP^2$, Nouh~\cite{MR3159966} showed that the adjunction inequality on $c = 3H_1+ 3H_2$ does not hold in general.
(In addition, by using blow-up, we can construct surfaces in $2\CP^2 \# n(-\CP^2)$ which violate the adjunction inequality on $c = 3H_1 + 3H_2 - \sum_{q=1}^n E_q$.)
To our knowledge, the adjunction-type inequalities for $2\CP^2 \# n(-\CP^2)$ is only Strle's~\cite{MR2064429} in previous researches.
For such $4$-manifold, Strle's genus bound is  the adjunction inequality for at least one of two disjoint surfaces with positive self-intersection numbers.
On the other hand, Theorem~\ref{main thm : special} is the adjunction inequality for at least one of four surfaces with self-intersection number zero and we allow certain geometric intersections.
In fact, any geometric intersections of $\Si_1$ and $\Si_3$, $\Si_2$ and $\Si_4$ are allowed in Theorem~\ref{main thm : special}.

The same remarks can be said for the general version of our main theorem (Theorem~\ref{general}) and its consequence (Corollary~\ref{cor : adj. ineq for a single surface : mCP^2}) which gives certain adjunction-type inequalities for $m\CP^2 \# n(-\CP^2)\ (m,n \geq 2)$.
In addition, we can give an alternative proof of Strle's result from the general version of our main theorem (Corollary~\ref{positive self-intersection}).
\end{rem}

Under certain assumptions on geometric intersections with embedded surfaces violating the adjunction inequalities, we can derive the adjunction inequality for a single surface.
(This will be generalized to the adjunction inequality for a single surface in $m\CP^2\#n\CP^2\ (m,n \geq 2)$ in Corollary~\ref{cor : adj. ineq for a single surface : mCP^2}).
Before starting to state it, we mention an easy method to make surfaces with small genera.
For a homology class $\beta = aH_2 + \sum_{q=1}^n b_q E_q \in H_2(\CP^2_2 \# n(-\CP^2);\Z)$,
considering algebraic curves $C \subset \CP^2_2$ and $C_q \subset \CP^2_q$ and reversing orientations of them if we need,
we can easily construct the surface $S \subset \CP^2_2 \# n(-\CP^2)$ by
\begin{align}
S := C \# (\#_{q=1}^n C_q)
\label{conn. sum construction}
\end{align}
satisfing
\[
[S] = \beta,\quad g(S) = \frac{(|a|-1)(|a|-2)}{2} + \sum_{q=1}^n \frac{(|b_q|-1)(|b_q|-2)}{2}.
\]
For example, for the characteristic $H_2-\sum_{q=1}^nE_q$, such naive construction is sufficient to give many examples of surfaces which violate the adjunction inequality.
In the following Corollary~\ref{cor : adj. ineq for a single surface}, we give the adjunction inequality for a single surface.
The above construction will produce many examples of surfaces satisfying the assumption in Corollary~\ref{cor : adj. ineq for a single surface}.

\begin{cor}
For the $4$-manifold
\[
X = 2\CP^2 \# n(-\CP^2)\quad (n>0),
\]
let $c \in H^2(X;\Z)$ be a characteristic with $c^2 > \sign(X)$ and $c \cdot H_1 > -3$.
Let $\alpha \in H_2(X;\Z)$ be a homology class with $\alpha^2 = 0$ and assume that
\[
(H_1 \cdot \alpha) (c \cdot \alpha) < 0.
\]
Let $\beta_i \in H_2(\CP^2_2 \# n(-\CP^2);\Z)\ (i = 1,2)$ be homology classes with $\beta_i^2 = 0$ and $S_i \subset \CP^2_2 \# n(-\CP^2) \setminus ({\rm disk}) \subset X$ be surfaces with $[S_i] = \beta_i$.
Assume that
\begin{align*}
H_2 \cdot \beta_i > 0,\quad (-1)^{i-1}c \cdot \beta_i > 0,\quad -\chi(S_i) < |c \cdot \beta_i| .
\end{align*}
Then, for an embedded surface $\Si \subset X$ satisfying $[\Si] = \alpha$ and $\Si \cap S_i = \emptyset\ (i = 1, 2)$, the inequality
\[
-\chi(\Si) \geq |c \cdot \alpha|
\]
holds.
\label{cor : adj. ineq for a single surface}
\end{cor}

\begin{proof}[Proof of Corollary~\ref{cor : adj. ineq for a single surface}]
Let $m := |c \cdot H_1| + 1$ and 
\[
X' := X \# m^2(-\CP^2) = X \# (\#_{q=n+1}^{n+m^2}(-\CP^2_q) ),\quad c' := c - \sum_{q=n+1}^{n+m^2}E_{q}.
\]
For the homology class $\gamma := mH_1 + \sum_{q=n+1}^{n+m^2}E_{q}$, we obtain
\begin{align*}
c' \cdot \gamma
= m c \cdot H_1 + m^2
= (|c \cdot H_1| + 1) c\cdot H_1 +(|c \cdot H_1| + 1)^2
> 0.
\end{align*}
Let $S \subset \CP^2_1 \# (\#_{q=n+1}^{n+m^2}(-\CP^2_{q})) \setminus ({\rm disk}) \subset X'$ be a surface with $[S] = \gamma$ obtained as (\ref{conn. sum construction}).
Note that $S \cap S_i = \emptyset\ (i=1,2)$ hold.
Since we have
\[
| c' \cdot \gamma | +2 = m c \cdot H_1 + m^2 + 2 > m^2 - 3m + 2 = 2g(S),
\]
the surface $S$ violates the adjunction inequality
\[
-\chi(S) \geq | c' \cdot \gamma |.
\]
Similarly, the surfaces $S_i\ (i=1,2)$ also violate the adjunction inequalities
\[
-\chi(S_i) \geq | c' \cdot \beta_i |
\]
by our assumptions.

Let regard
\begin{align*}
X',\quad c',\quad \alpha,\quad \beta_1,\quad  \gamma,\quad  \beta_2
\end{align*}
as 
\begin{align*}
X ,\quad c ,\quad \alpha_1 ,\quad \alpha_2 ,\quad \alpha_3,\quad \alpha_4 
\end{align*}
in Theorem~\ref{main thm : special}.
Then it is easy to see that they satisfy (A) in Theorem~\ref{main thm : special}.
(See Figure~\ref{figure : Proof of Corollary}.)
If $\Si \cap S_i = \emptyset\ (i=1, 2)$ hold, the surfaces $\Si, S_1, S$ and $S_2$ satisfy (B) in Theorem~\ref{main thm : special}.
Thus we can derive the conclusion by Theorem~\ref{main thm : special}.
\end{proof}

\begin{center}
\begin{figure}
\begin{tikzpicture}
\draw(2.6,0) node {$x_1$};
\draw(0,2.5) node {$x_2$};
\draw[->] (-2.3,0) -- (2.3, 0);
\draw[->] (0,-2.3)-- (0,2.3);
\draw (-2,-2)-- (1.3,2);
\draw (0.5,-2)-- (0.5,2);
\draw (-2, 1.6)-- (2, 1.6);
\draw (-2, -1.3)-- (2, -1.3);
\draw(-1.1,-0.4) node {$L_1$};
\draw(0.8,1.85) node {$L_2$};
\draw(0.8,-0.4) node {$L_3$};
\draw(-0.3,-1.6) node {$L_4$};
\fill(0,0) circle (2pt);
\end{tikzpicture}
\caption{Proof of Corollary~\ref{cor : adj. ineq for a single surface}}
\label{figure : Proof of Corollary}
\end{figure}
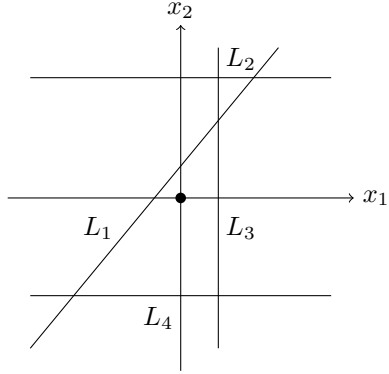
\end{center}

\begin{ex}
Let us give natural numbers $d_1 \geq 4,\ d_2 \geq 1,\ d_3 \geq 2$ and $n \geq d_1^2 + \max\{ d_2^2, d_3^2\}$.
For $X = 2\CP^2 \# n(-\CP^2)$,
let us consider the homology classes
\begin{align*}
\alpha &:= d_1H_1 - \sum_{q=1}^{d_1^2}E_q,\\
\beta_1 &:= d_2H_2 + \sum_{q=d^2_1+1}^{d_1^2+d_2^2} E_q,\\
\beta_2 &:= d_3H_2 - \sum_{q=d^2_1+1}^{d_1^2+d_3^2} E_q.
\end{align*}
Let $S_i \subset \CP^2_2 \# n(-\CP^2) \setminus ({\rm disk}) \subset X$  be surfaces with $[S_i] = \beta_i$ obtained as (\ref{conn. sum construction}).
For the characteristic $c = 3H_1 + H_2 -\sum_{q=1}^n E_q$,
it is easy to see that $\alpha,\ \beta_i$ and $c$ satisfy the assumptions of Corollary~\ref{cor : adj. ineq for a single surface}.
Thus, for an embedded surface $\Si \subset X$ satisfying $[\Si] = \alpha$ and $\Si \cap S_i = \emptyset\ (i=1,2)$, we have
\begin{align}
g(\Si) \geq \frac{(d_1-1)(d_1-2)}{2}
\label{ineq : example for a single surface like Thom conj.}
\end{align}
by Corollary~\ref{cor : adj. ineq for a single surface}.

By the adjunction formula for $\CP^2_1 \# n(-\CP^2)$, the homology class $\alpha$ can be represented by a surface $\Si$ of genus $(d_1-1)(d_1-2)/2$ satisfying $\Si \cap S_i = \emptyset$.
Thus the inequality (\ref{ineq : example for a single surface like Thom conj.}) is the optimal bound under the condition $\Si \cap S_i = \emptyset$.
\label{ex : example for a single surface like Thom conj.}
\end{ex}

\begin{ex}
Let give natural numbers $d_1 \geq1,\ d_2 \geq 0,\ d_3, d_4 \geq 2$ and $n$ satisfying
\begin{align*}
d_3 \geq d_4 \geq \max\{d_2,2\}  ,\ d_1^2+d_2^2 - 2d_2d_3 - 3d_1 - d_2 > 0,\\
n \geq d_1^2+d_2^2+d_3^2+d_4^2-d_2d_3-d_2d_4.
\end{align*}
For $X = 2\CP^2 \# n(-\CP^2)$,
let us consider the homology classes
\begin{align*}
\alpha &:= d_1H_1 + d_2H_2 + \sum_{q=1}^{d_2d_3}E_q - \sum_{q=d_3^2+1}^{ d_3^2+d_2d_4}E_q - \sum_{q = d_3^2+d_4^2+1}^{d_1^2+d_2^2+d_3^2+d_4^2-d_2d_3-d_2d_4}E_q,\\
\beta_1 &:= d_3H_2 + \sum_{q=1}^{d_3^2} E_q,\\
\beta_2 &:= d_4H_2 - \sum_{q=d^2_3+1}^{d_3^2 + d_4^2} E_q.
\end{align*}
Let $S_i \subset \CP^2_2 \# n(-\CP^2) \setminus ({\rm disk}) \subset X$  be surfaces with $[S_i] = \beta_i$ obtained as (\ref{conn. sum construction}).
For the characteristic $c = 3H_1 + H_2 -\sum_{q=1}^n E_q$,
it is easy to check that $\alpha,\ \beta_i$ and $c$ satisfy the assumptions of Corollary~\ref{cor : adj. ineq for a single surface}.
Thus, for an embedded surface $\Si \subset X$ satisfying $[\Si] = \alpha$ and $\Si \cap S_i = \emptyset\ (i=1,2)$, we have
\begin{align}
g(\Si) \geq \frac{(d_1-1)(d_1-2)}{2} + \frac{d_2}{2} (d_2 -2d_3 -1)
\label{ineq : example for a single surface not like Thom conj.}
\end{align}
by Corollary~\ref{cor : adj. ineq for a single surface}.

If $d_2 = 0$, this example is quite similar to Example~\ref{ex : example for a single surface like Thom conj.}.
Thus, by the same way in Example~\ref{ex : example for a single surface like Thom conj.}, the bound (\ref{ineq : example for a single surface not like Thom conj.}) is optimal under the condition $\Si \cap S_i = \emptyset\ (i=1,2)$ in the case of $d_2 = 0$.
However, for $d_2 > 0$, we have no answer to this realization problem: 
What kinds of $d_1, \ldots, d_4$ does there exist an embedded surface $\Si \subset X$satisfying
\[
g(\Si) = \frac{(d_1-1)(d_1-2)}{2} + \frac{d_2}{2} (d_2 -2d_3 -1)
\]
and $\Si \cap S_i = \emptyset\ (i=1,2)$ for?

In the same way, one can produce a lot of new realization problems from Corollary~\ref{cor : adj. ineq for a single surface}.
\end{ex}

\subsection{General form of the main theorem}
\label{subsection : General form of the main theorem}

In this subsection, we formulate the most general form of our main theorem and give the adjunction inequality for a single surface in $m\CP^2\# n(-\CP^2)$.
We also give an alternative proof of Strle's adjunction inequalities.

We first describe a certain condition on embedded surfaces and a cohomology class which is a generalization of the conditions (A) and (B) in Theorem~\ref{main thm : special}.
Let $X$ be an oriented closed smooth $4$-manifold, $\Si_1, \ldots, \Si_k \subset X$ be embedded surfaces with zero self-intersection number, and $c \in H^2(X;\Z)$ be a cohomology class on $X$.
Set $\alpha_i = [\Si_i]$ and assume that $c \cdot \alpha_i \neq 0$ holds for each $i =1, \ldots , k $.
The set
\[
\mathcal{S} := \Set{I\subset\{1,\ldots ,k\} | 
I\neq\emptyset,\ \{\Si_i\}_{i\in I}\ {\rm are\ disjoint.}}
\]
has a structure of an abstract simplicial complex.
Consider the vector space
\[
\mathcal{V} := \R v_1 \oplus \cdots \oplus \R v_k
\]
generated by the vertices $v_1 = \{1\}, \ldots, v_k = \{k\}$ of the simplicial complex.
We denote by $\mathcal{P} \subset \mathcal{V}^\ast = \Hom(\mathcal{V},\R)$ the set
\[
\mathcal{P} :=
\Set{
\varphi\in\mathcal{V}^\ast | \begin{matrix}
\left<\varphi,v_i\right>\geq 0\ (1\leq i\leq k),\\
\Set{i\in\{1,\ldots,k\}|\left<\varphi,v_i\right>>0}\in\mathcal{S}\cup\{\emptyset\}
\end{matrix}
 }.
\]
($\mathcal{P}$ will be used as a ``parameter space" when we consider a family of metrics of $X$ in the proof of the main theorem.)
Let $Q$ be the intersection form of $X$.
We denote $b^+ = b^+(X)$ the maximal dimension of positive definite subspace with respect to $Q$.
Fix a $b^+$-dimensional positive definite (with respect to $Q$) subspace $V^+ \subset H^2(X;\R)$.
We denote by $V^-$ the orthogonal complement of $V^+$ with respect to $Q$.
Note that $V^-$ is a negative definite subspace.
Let $p_{V^+} : H^2(X; \R) = V^+ \oplus V^- \to V^+$ be the projection.
For each $\alpha_i$, we define the hyperplane $\Hyp_{\alpha_i}$ in $(V^+)^\ast = \Hom(V^+,\R)$ by
\begin{align*}
\Hyp_{\alpha_i} := \Set{f \in \Hom(V^+,\R) | f(p_{V^+}(\alpha_i)) = c\cdot\alpha_i}.
\end{align*}
(In the case of $X = 2\CP^2 \# n(-\CP^2)$, this hyperplane $\Hyp_{\alpha_i}$ corresponds to the line $L_i$ defined as (\ref{line L_i}).)
Note that $0 \notin \Hyp_{\alpha_i}$ since $c \cdot \alpha_i \neq 0$.

For a continuous map $F : \mathcal{P} \to (V^+)^\ast$, if there exist positive numbers $R_i>0$ such that
\begin{align}
F(\{\varphi\in \mathcal{P}\ |\ \left<\varphi,v_i\right>\geq R_i\}) \subset \Hyp_{\alpha_i},
\label{mapstoplane}
\end{align}
then there exists a compact subset $\mathcal{K} \subset \mathcal{P}$ such that $F(\mathcal{P} \setminus \mathcal{K}) \subset (V^+)^\ast \setminus \{0\}$ since $0 \notin \Hyp_{\alpha_i}$.
Thus a continuous map $F : \mathcal{P} \to (V^+)^\ast$ satisfying (\ref{mapstoplane}) induces
\[
F^\ast : H^\ast((V^+)^\ast,(V^+)^\ast \setminus \{0\};\Z) \to H^\ast(\mathcal{P}, \mathcal{P} \setminus \mathcal{K};\Z).
\]
Then we obtain the map
\[
F^\ast_{\rm cpt} : H^{b^+}((V^+)^\ast,(V^+)^\ast \setminus \{0\} ; \Z) \to H^{b^+}_c(\mathcal{P} ; \Z)
\]
by the composition of $F^\ast$ and the natural map $H^{b^+}(\mathcal{P}, \mathcal{P} \setminus \mathcal{K} ; \Z)\to H^{b^+}_c(\mathcal{P} ; \Z)$, where $H_c^\ast(\cdot)$ means the cohomology with compact supports.

Here we consider a condition for $\Si_1, \ldots, \Si_k$ and $c$.
In our main theorem, we will assume that $\Si_1, \ldots, \Si_k$ and $c$ satisfy the following Condition~\ref{main condition} consisting of two conditions.

\begin{condition}1
\begin{description}
\item[(i)] There exists a continuous map $\R^{b^+} \to \mathcal{P}$ which induces an isomorphism $H^{b^+}_c(\mathcal{P};\Z) \simeq H^{b^+}_c(\R^{b^+};\Z)\ (\simeq \Z)$.
\item[(ii)] If a continuous map $F : \mathcal{P} \to (V^+)^\ast$ satisfies (\ref{mapstoplane}) for sufficiently large positive numbers $R_i > 0\ (1 \leq i\leq k)$, then the mapping degree of $F:\mathcal{P}\to(V^+)^\ast$ is not zero, i.e., $F_{\rm cpt}^\ast$ is a non-trivial map.
\end{description}
\label{main condition}
\end{condition}

\begin{rem}
Here we give a description of $\Hyp_{\alpha_i}$ in Condition~\ref{main condition} when we use a basis of $V^+$.
Let $u_1, \ldots, u_{b^+}$ be a basis of $V^+$.
We define $a_{i,j}\in\R$ as $p_{V^+}(\alpha_i) = \sum_{j=1}^{b^+}a_{i,j}u_j$.
For each $i$, consider the hyperplane
\begin{align*}
\Hyp_{\alpha_i}^u := \Set{(x_1,\ldots,x_{b^+})\in\R^{b^+} | \sum_{j=1}^{b^+}a_{i,j}x_j = c \cdot \alpha_i}
\end{align*}
in $\R^{b^+}$.
This corresponds to $\Hyp_{\alpha_i}$ in Condition~\ref{main condition}.
\label{rem : description of main condition using basis}
\end{rem}

\begin{rem}
Here we explain that Condition~\ref{main condition} is a generalization of the combination of the conditions (A) and (B) in Theorem~\ref{main thm : special}.
In the situations of Theorem~\ref{main thm : special},
$\mathcal{P}$ is homeomorphic to $\R^2$.
Using $H_1, H_2$ as $u_1, u_2$ in Remark~\ref{rem : description of main condition using basis},
$\Hyp_{\alpha_i}$ corresponds to the line $L_i$.
Let assume that $\Si_i$'s and $c$ satisfy the conditions (A).
Suppose that a map $F : \mathcal{P} \to (V^+)^\ast$ satisfies (\ref{mapstoplane}) for sufficiently large positive numbers $R_i > 0$.
Then $F$ maps a large sphere $S(\R^2) \simeq S^1$ in $\R^2  (\simeq \mathcal{P})$ on a small neighborhood of the quadrilateral consisting of $L_i$'s with degree $\pm1$.
This means the mapping degree of $F:\mathcal{P}\to(V^+)^\ast$ is equal to $\pm1$.
Thus (ii) in Condition~\ref{main condition} is a generalization of the condition (A).
\label{rem : main condition is a generalization of conditions in the first theorem}
\end{rem}

\begin{rem}
If we take another $b^+$-dimensional positive definite subspace $W^+ \subset H^2(X;\R)$ instead of $V^+$, then $\Si_1, \ldots, \Si_k$ and $c$ satisfy Condition~\ref{main condition} if and only if the corresponding condition for $W^+$ holds.
(Indeed, two $b^+$-dimensional positive definite subspaces can be connected by an isotopy in $H^2(X;\R)$, and Condition~\ref{main condition} is the statement for only mapping degree, so the claim follows from the homotopy invariance of mapping degree.)
\end{rem}

The following theorem is the most general form of the main theorem.

\begin{theo}
Let $X$ be an oriented closed smooth $4$-manifold with $b_1(X) = 0$, $c\in H^2(X;\Z)$ be a characteristic with $c^2 > \sign(X)$, and $\alpha_1, \ldots, \alpha_k \in H_2(X;\Z)\ (k>b^{+}(X))$ be homology classes with 
\[
\alpha^2_i=0,\quad c \cdot \alpha_i \neq 0\quad (i=1,\ldots,k).
\]
For embedded surfaces $\Si_1, \ldots, \Si_{k} \subset X$ with $[\Si_i] = \alpha_i\ (1 \leq i \leq k)$, assume that $\Si_1, \ldots, \Si_{k}$ and $c$ satisfy Condition~\ref{main condition}. 
Then the inequality
\begin{align}
-\chi(\Si_i) \geq |c\cdot\alpha_i|
\label{pre-adj-ineq}
\end{align}
holds for at least one $i \in \{1,\ldots,k\}$.
\label{general}
\end{theo}

The proof of Theorem~\ref{general} will be given in Section~\ref{section : Proof of the main theorem}.
From Remark~\ref{rem : main condition is a generalization of conditions in the first theorem}, Theorem~\ref{general} is a generalization of Theorem~\ref{main thm : special}.

\begin{rem}
One can eliminate the assumption $b_1(X) = 0$ in the statement of Theorem~\ref{general} by surgeries along generators of $H_1(X;\Z)$.
In fact, we can assume that loops which represents generators of $H_1(X;\Z)$ do not intersects with $\Si_i$'s by moving the loops along the fiber of the normal bundles of $\Si_i$'s.

For simplicity, we only consider the case of $b_1(X) = 0$ in this paper.
\end{rem}

From Theorem~\ref{general}, we can obtain the adjunction inequalities for $m\CP^2 \# n(-\CP^2)$ $(m,n \geq 2)$ under certain assumptions on geometric intersections.
Let $m \geq 2,n_p \geq 1\ (p = 1, \ldots, m)$ and $n = \sum_{p=1}^m n_p$.
For the $4$-manifold $X = m\CP^2 \# n(-\CP^2)$, we write $X = \#_{p=1}^m X_p$, where
\[
X_p = \CP^2_p \# n_p (-\CP^2) = \CP^2_p \# (\#_{q=1}^{n_p} (-\CP^2_{p,q})).
\]
Let $H_p$ denote a generator of $H_2(\CP^2_p;\Z)$.

\begin{cor}
For the $4$-manifold
\begin{align*}
X = m\CP^2 \# n(-\CP^2) \quad (m,n \geq 2),
\end{align*}
let $c \in H^2(X;\Z)$ be a characteristic with $c^2 > \sign(X)$ and $c \cdot H_1 > -3$.
Let $\alpha \in H_2(X;\Z)$ be a homology class with $\alpha^2 = 0$ and assume that
\[
(H_1 \cdot \alpha) (c \cdot \alpha) < 0.
\]
Let $\beta_{p, i} \in H_2(X_p;\Z)\ (p = 2, \ldots, m,\ i = 1,2)$ be homology classes with $\beta_{p, i}^2 = 0$ and $S_{p,i} \subset X_p \setminus ({\rm disk}) \subset X$ be surfaces with $[S_{p,i}] = \beta_{p,i}$.
Assume that
\begin{align*}
H_p \cdot \beta_{p, i} > 0,\quad (-1)^{i-1}c \cdot \beta_{p,i} > 0,\quad -\chi(S_{p,i}) < |c \cdot \beta_{p,i}|.
\end{align*}
Then, for an embedded surface $\Si \subset X$ satisfying $[\Si] = \alpha$ and $\Si \cap S_{p, i} = \emptyset\ (p=2,\ldots,m,\ i = 1, 2)$, the inequality
\[
-\chi(\Si) \geq |c \cdot \alpha|
\]
holds.
\label{cor : adj. ineq for a single surface : mCP^2}
\end{cor}

\begin{proof}[Proof of Corollary~\ref{cor : adj. ineq for a single surface : mCP^2}]
Recall the arguments in the proof of Corollary~\ref{cor : adj. ineq for a single surface}.
We have $\mathcal{P} \simeq \R^{m}$ in this case and we consider the $(m-1)$-dimensional ``polytope" instead of the quadrilateral in \S~2.
We can obtain $\gamma$ as in the proof of Corollary~\ref{cor : adj. ineq for a single surface}.
By the same argument in Remark~\ref{rem : main condition is a generalization of conditions in the first theorem}, it is enough to show that the ``polytope" obtained from $\alpha$, $\beta_{p,i}$'s and $\gamma$ is a bounded ``polytope" including the origin of $\R^m$ in its interior.
Then the proof is definitely similar to the proof of Corollary~\ref{cor : adj. ineq for a single surface}.
\end{proof}

Here we give a generalization of Example~\ref{ex : example for a single surface like Thom conj.}.
Let $E_{p,q}$ denote a generator of $H_2(-\CP^2_{p,q};\Z)$.

\begin{ex}
Let give natural numbers $m \geq 2,\ d_1 \geq 4,\ n_1 \geq d_1^2$ and
\[
d_{p,2} \geq 1,\quad d_{p,3} \geq 2,\quad n_p \geq \max\{ d_{p,2}^2, d_{p,3}^2\}\quad(p=2,\ldots ,m).
\]
Put $n = \sum_{p=1}^m n_p$.
For $X = m\CP^2 \# n(-\CP^2)$,
let us consider the homology classes
\begin{align*}
\alpha &:= d_1H_1 - \sum_{q=1}^{d_1^2}E_{1,q},\\
\beta_{p,1} &:= d_{p,2}H_p + \sum_{q = 1}^{d_{p,2}^2} E_{p,q} \quad (p=2,\ldots, m),\\
\beta_{p,2} &:= d_{p,3}H_p - \sum_{q = 1}^{d_{p,3}^2} E_{p,q} \quad (p=2,\ldots, m).
\end{align*}
For $p = 2, \ldots ,m$ and $i = 1, 2$, let $S_{p,i} \subset X_p \setminus ({\rm disk}) \subset X$  be surfaces with $[S_{p,i}] = \beta_{p,i}$ obtained as (\ref{conn. sum construction}).
For the characteristic $c = 3H_1 + \sum_{p=2}^m H_p -\sum_{p=1}^m\sum_{q=1}^n E_{p,q}$,
it is easy to see that $\alpha,\ \beta_{p,i}$ and $c$ satisfy the assumptions of Corollary~\ref{cor : adj. ineq for a single surface}.
Thus, for an embedded surface $\Si \subset X$ satisfying $[\Si] = \alpha$ and $\Si \cap S_{p, i} = \emptyset$ $(p = 2, \ldots, m,\ i = 1, 2)$, we have
\begin{align}
g(\Si) \geq \frac{(d_1-1)(d_1-2)}{2}
\label{ineq : example for a single surface like Thom conj. general}
\end{align}
by Corollary~\ref{cor : adj. ineq for a single surface : mCP^2}.

In the same way in Example~\ref{ex : example for a single surface like Thom conj.}, the inequality (\ref{ineq : example for a single surface like Thom conj. general}) is the optimal bound under the condition $\Si \cap S_{p,i} = \emptyset$ $(p=2,\ldots,m,\ i = 1, 2)$.
\end{ex}

In Strle~\cite{MR2064429}, he showed that the adjunction inequality holds for at least one of disjoint $b^+$ surfaces with positive self-intersection numbers.
On the other hand, Dai-Ho-Li~\cite{MR3465838} derived an alternative simple proof of Strle's theorem in the case of $b^+=1$.
Here we give an alternative proof of Strle's adjunction inequality for any $b^+$.

\begin{cor}(Strle~\cite{MR2064429} Theorem A, Theorem B. See also Dai-Ho-Li~\cite{MR3465838} Theorem 1.4.)
For a surface $\Si$, put
\[
\chi^-(\Si) := \max\{-\chi(\Si),0\}.
\]
Let $X$ be an oriented closed smooth $4$-manifold with $b_1(X) = 0$ and  $c \in H^2(X;\Z)$ be a characteristic  with $c^2 > \sign(X)$.
\begin{description}
\item[(A)] In the case of $b^+(X)=1$, let $\alpha \in H_2(X;\Z)$ be a homology class with $\alpha^2 > 0$ and  $\Si \subset X$ be an embedded surface with $[\Si] = \alpha$.
Then the inequality
\begin{align}
\chi^-(\Si) \geq -|c\cdot\alpha| + \alpha^2
\end{align}
holds.\\
\item[(B)] In the case of $b^+(X)>1$, let $\alpha_1, \ldots, \alpha_{b^+} \in H_2(X;\Z)$ be homology classes with $\alpha_i^2 > 0\ (1 \leq i \leq b^+)$ and $\Si_1, \ldots, \Si_{b^+} \subset X$ be embedded surfaces with $[\Si_i] = \alpha_i$.
Assume that $\Si_1, \ldots, \Si_{b^+}$ are disjoint.
Then the inequality
\begin{align}
\chi^- (\Si_i) \geq - |c\cdot\alpha_i| + \alpha_i^2
\label{ThmB}
\end{align}
holds for at least one $i \in \{1,\ldots,b^+\}$.
\end{description}
\label{positive self-intersection}
\end{cor}

\begin{rem}
Strle~\cite{MR2064429} and Dai-Ho-Li~\cite{MR3465838} showed that, when $g(\Si)=0$, a sharper result
\[
-2 = -\chi(\Si) \geq - |c\cdot\alpha| + \alpha^2
\]
holds.
Dai-Ho-Li used results in Morgan-Szab\'{o}-Taubes~\cite{MR1438191} to treat this case.
\label{Strle and Dai-Ho-Li}
\end{rem}

\begin{proof}[Proof of Corollary~\ref{positive self-intersection}]
The claim {\bf (A)} can be regarded as a special case of {\bf (B)} by putting $\alpha=\alpha_1$ and $\Si=\Si_1$, so we prove only {\bf (B)}.
Let $n_i := \alpha_i^2( > 0)$.
We may assume that $|c\cdot\alpha_i| < n_i$ holds for any $i = 1, \ldots, n$.
Let $X'$ be
\[
X' := X \# n_1 (-\CP^2) \# \cdots \# n_{b^+} (-\CP^2)
\]
and write
\[
n_i(-\CP^2) = \#_{q=1}^{n_i}(-\CP^2_{i,q}).
\]
Let $E_{i, q}$ be a generator of $H^2(-\CP^2_{i,q};\Z)$.
Set $E^i := E_{i, 1} + \cdots + E_{i, n_i}$ and $c' := c - \sum_{i=1}^{b^+} E^i \in H^2(X';\Z)$.

For each $i$, let us denote by $\Si_i^\pm\subset X'$ for $\Si_i^\pm:=\Si_i\#n_i\CP^1$, where the orientation of the connected sum is taken as 
\[
\alpha_i^\pm := [\Si_i^\pm]=\alpha_i\pm E^i.
\]
For all $i$
\begin{align}
c' \cdot \alpha_i^+ = c \cdot \alpha_i + n_i > 0,\quad c' \cdot \alpha_i^- = c \cdot \alpha_i - n_i < 0
\label{sign}
\end{align}
hold since we assumed that $|c\cdot\alpha_i| < n_i$.
Note that the self-intersection number of any $\Si_i^\pm$ is zero. 

Let take a $b^+$-dimensional positive definite subspace
\[
V^+ := {\rm span}\Set{\alpha_1,\ldots,\alpha_{b^+}} \subset H^2(X;\R)
\]
and fix $u_1 = \alpha_1,\ldots, u_{b^+} = \alpha_{b^+}$ as a basis of $V^+$.
Then
\[
p_{V^+}(\alpha_i^\pm) = \alpha_i\ (i =1, \ldots, b^+)
\]
holds for the projection $p_{V^+} : H^2(X;\R) = V^+ \oplus V^- \to V^+$.
Dividing $\R^{b^+}$ by the hyperplane
\begin{align*}
\Hyp_{\alpha_i^\pm}^u = \Set{ (t_1, \ldots, t_{i-1}, c' \cdot \alpha_i^\pm, t_{i+1}, \ldots, t_{b^+}) | t_1, \ldots, t_{i-1}, t_{i+1}, \ldots, t_{b^+} \in \R},
\end{align*}
we have two half spaces ; we denote by $\Half_{\alpha_i^\pm}$ the half space which includes the origin of $\R^{b^+}$.
Then by (\ref{sign}), the polytope
\[
P = 
\bigcap_{\substack{1 \leq i \leq b^+,\\ \epsilon_i \in \{+, -\}}}  \Half_{\alpha_i^{\epsilon_i}}^u
\]
is a cuboid including the origin of  $\R^{b^+}$ in its interior.

On the other hand,  if $j_1 \neq j_2$, we have $\Si_{j_1}^{\epsilon_{j_1}} \cap \Si_{j_2}^{\epsilon_{j_2}} = \emptyset$ for all $\epsilon_{j_1}, \epsilon_{j_2} \in \{+, -\}$.
Thus, for the simplicial complex $\mathcal{S}$ obtained by $\Si_1^\pm, \ldots, \Si_{b^+}^\pm$, we have an isomoprhism
\[
\mathcal{S} \simeq \Set{ \Set{ \Si_{j_1}^{\epsilon_1}, \ldots , \Si_{j_i}^{\epsilon_i}} | 1\leq i \leq b^+,\ \epsilon_1, \ldots, \epsilon_i \in \{+, -\},\  j_1, \ldots, j_{i}\ {\rm are\ distinct.} }
\]
as simplicial complexes.
Thus $\mathcal{P}$ is homeomorphic to $\R^{b^+}$.
Let $v_i^\pm\ (1\leq i \leq b^+)$ denote the vertices of $\mathcal{S}$.
Then, for large numbers $R_i > 0\ (1 \leq i \leq b^+)$, the space
\[
\bigcap_{\substack{1 \leq i \leq b^+,\\ \epsilon_i \in \{+, -\}}} \{\varphi\in \mathcal{P}\ |\ \left<\varphi,v_i^{\epsilon_i} \right> < R_i\}
\]
naturally has a structure of polytope and is not only homeomorphic but also equivalent to the cuboid $P$ as simplicial complexes.
Therefore it is easy to see that $\Si_1^\pm, \ldots, \Si_{b^+}^\pm$ and $c'$ satisfy Condition~\ref{main condition}.
Thus we can apply Theorem~\ref{general}, then the inequality corresponding (\ref{pre-adj-ineq}) holds for at least one of $\Si_1^\pm, \ldots, \Si_{b^+}^\pm$.
If the inequality holds for $\Si_i^+$, we have
\[
\chi^-(\Si_i) \geq c \cdot \alpha_i + \alpha_i^2,
\]
and if the inequality holds for $\Si_i^-$, we have
\[
\chi^-(\Si_i) \geq -c \cdot \alpha_i + \alpha_i^2.
\]
Hence in both cases, we have (\ref{ThmB}) for $i$.
\end{proof}

\section{Proof of the main theorem}
\label{section : Proof of the main theorem}

In this section, we prove Theorem~\ref{general} assuming Lemma~\ref{stretch} and Proposition~\ref{analysis} which will be shown in the next section.

Let $X, \alpha_1, \ldots, \alpha_k, \Sigma_1, \ldots, \Sigma_k$ and $c$ be the one in the statement of Theorem~\ref{general}.
The key to the proof is to construct a family of metrics of $X$ by stretching the neighborhood of surfaces $\Si_1, \ldots, \Si_k$ and to describe positions of metrics in relation to the wall in terms of several embedded surfaces.
This construction of the family of metrics is a slight generalization of one due to Fr{\o}yshov~\cite{MR2052970} used in the context of instanton Floer homology.

We define a continuous injection $\iota : \mathcal{P} \to \Met(X)$ as follows, where $\Met(X)$ is the space of Riemannian metrics on $X$.

Fix a metric $g_0$ on $X$.
For each surface $\Si_i$, let us consider the normal bundle $\nu_i \to \Si_i$.
We identify the total space $\nu_i$ with a neighborhood of $\Si_i$ in $X$.
For the sphere bundle $S(\nu_i)$, there exists a neighborhood $U_i$ in $X$ which is diffeomorphic to $[0,1] \times S^1 \times \Si_i$.
By taking sufficiently small neighborhood, we may assume that $U_i \cap U_j = \emptyset$ if $\Sigma_i \cap \Sigma_j = \emptyset$.
We write $V_i$ for the neighborhood of $\Si_i$ which corresponds to $[1/3,2/3] \times S^1 \times \Si_i$ via the diffeomorphism $U_i \simeq [0,1] \times S^1 \times \Si_i$.
For $\varphi \in \mathcal{P}$, let $\mathscr{S}(\varphi)$ denote the set
\[
\Set{i \in \{1,\ldots,k\} | \left<\varphi, v_i\right> > 0} \in \mathcal{S} \cup \{\emptyset\}.
\]

\begin{enumerate}
\item In the case of $\mathscr{S}(\varphi) = \emptyset$, we define $\iota(\varphi) = g_0$.
\item In the case of $\mathscr{S}(\varphi) \neq \emptyset$, recall that $\{\Si_i\}_{i\in\mathscr{S}(\varphi)}$ are disjoint.
For each $i \in \mathscr{S}(\varphi)$, we will give the following metric $\iota(\varphi;U_i)$ on $U_i$, and we define $\iota(\varphi)$ as the metric on $X$ obtained by gluing $\iota(\varphi;U_i)$ and $g_0|_{X \setminus \coprod_{i\in\mathscr{S}(\varphi)} V_i}$ by a partition of unity.
\begin{enumerate}
\item When $\left<\varphi, v_i\right> \leq1$, connect $g_0|_{U_i}$ and the product metric $g_{[0,1] \times S^1\times\Si_i}$ on $U_i$ by a line, and let $\iota(\varphi;U_i)$ be the metric on $U_i$ in the time of $\left<\varphi, v_i\right>$ in the line.
\item When $\left<\varphi, v_i\right> \geq 1$, let $\iota(\varphi;U_i)$ be the metric obtained by $\left<\varphi, v_i\right>$-scalar expansion for $[0,1]$-direction of the product metric $g_{[0,1]\times S^1 \times \Si_i}$ on $U_i$.
\end{enumerate}
\end{enumerate}

Here we use a common partition of unity in (2) for any $\varphi$.
Then $\iota$ is a continuous map.
For $i \in \mathscr{S}(\varphi)$, when $\left<\varphi, v_i\right> \geq 1$, we call the metric $\iota(\varphi)$ ``the metric stretched in the neighborhood of $\Si_i$ by the length $\left<\varphi, v_i\right>$".

\begin{rem}
If we consider only genus bounds for embedded surfaces, without loss of generality, we may assume that all surfaces intersect transversely by a deformation using an isotopy which keeps surfaces which are originally disjoint being disjoint.
Under the assumption, if $\Si_i \cap \Si_j \neq \emptyset$, one can obtain a metric which is the product metric around both $\Si_i$ and $\Si_j$ by taking the metric around each intersection point $p \in \Si_i \cap \Si_j$ as $\Si_i$ and $\Si_j$ intersect orthogonally.
If we choose the initial metric $g_0$ so that it is the product metric around the neighborhoods of all surfaces by this construction, we can define the continuous injection $\iota : \mathcal{P} \to \Met(X)$ in the following simple way:

For $\varphi\in \mathcal{P}$,
\begin{enumerate}
\item In the case of $\mathscr{S}(\varphi) = \emptyset$, we define $\iota(\varphi) = g_0$.
\item In the case of $\mathscr{S}(\varphi) \neq \emptyset$, let $\iota(\varphi)$ be the metric stretched in the neighborhood of each $\Si_i$ by the length $\left<\varphi, v_i\right>$ from the initial metric $g_0$.
\end{enumerate}

\end{rem}

For a metric $g$ on $X$, let $\Harm^g(X)$ and $\Harm^{+_g}(X)$ be the space of harmonic 2-forms and the space of harmonic self-dual 2-forms respectively.
Let us denote by $h_g : H^2(X;\R) \to \Harm^g(X)$ for the isomorphism defined by Hodge theory .
We wirte $\varphi_g : \Harm^{+_g}(X) \to V^+$ for the composition of the isomorphism $h_g^{-1}|_{\Harm^{+_g}(X)} : \Harm^{+_g}(X) \to H^2(X;\R)$ and the projection $p_{V^+} : H^2(X;\R) = V^+\oplus V^- \to V^+$.
Note that $\varphi_g:\Harm^{+_g}(X)\to V^+$ is a linear isomorphism since $\Ker\varphi_g \simeq V^+ \cap\Ker p_{V^+} = V^+ \cap V^- = \{0\}$.
We define a Euclidean metric on $H^2(X;\R) = V^+ \oplus V^-$ by the intersection on $V^+$ and $-1$ times intersection on $V^-$.

The following lemma will be proved in the next section.
\begin{lem}
Let $l>1$ and $\Si_1, \ldots, \Si_l \subset X$ be disjoint surfaces with zero self-intersection number.
For a metric $g$ on $X$, we write
\begin{align}
\omega_g^i := \varphi_g^{-1}(p_{V^+}(\alpha_i)) \in \Harm^{+_g}(X)\ (1 \leq i \leq l).
\label{harm}
\end{align}
For any positive numbers $R_1, \ldots, R_l > 0$, 
\begin{align}
|\alpha_i - [\omega^i_{g_{(R_1,\ldots,R_l)}}]|^4
\leq \frac{C\cdot\Area({\Si_i})}{\Area([0,R_i]\times S^1)}\ (1 \leq i \leq l)
\end{align}
hold for any metric $g_{(R_1, \ldots, R_l)}$ on $X$ stretched in the neighborhood of $\Si_i$ by the length $R_i$ from the initial metric $g_0$.
(Here it is not necessary that $g_{(R_1, \ldots, R_l)}$ coincides with the initial metric $g_0$ on the complement of the neighborhoods of $\Si_1,\ldots,\Si_l$.)
Therefore for each $i$ we have
\[
\lim_{R_i\to\infty} [\omega^i_{g_{(R_1, \ldots, R_l)}}] = \alpha_i,
\]
where the limit is uniform with respect to $R_1, \ldots, R_{i-1}, R_{i+1}, \ldots, R_l$. 
\label{stretch}
\end{lem}

Let $p_{+_g} : \Harm^g(X) \to \Harm^{+_g}(X)$ be the projection to the self-dual part with respect to a metric $g$.
We write $p_{+_g}(c)$ for $p_{+_g}(h_g(c))$.
Let $\mathcal{F} : \mathcal{P} \to (V^+)^\ast$ be the composition of $\iota : \mathcal{P} \to \Met(X)$ and
\begin{align}
\Met(X) \to (V^+)^\ast\ ; \  g\mapsto (v\mapsto [p_{+_g}(c)] \cdot [\varphi_g^{-1}(v)] = c\cdot [\varphi_g^{-1}(v)]),
\label{section}
\end{align}

In the proof of the following lemma, we use Condition~\ref{main condition}.

\begin{lem}
Assume that $\Si_1, \ldots, \Si_k$ and $c$ satisfy Condition~\ref{main condition}.
Then there exists a compact subset $\mathcal{K} \subset \mathcal{P}$ such that $\mathcal{F}(\mathcal{P} \setminus \mathcal{K}) \subset(V^+)^\ast \setminus \{0\}$.
(Hence we can define the mapping degree of $\mathcal{F} : \mathcal{P} \to (V^+)^\ast$.)
Moreover, the mapping degree of $\mathcal{F} : \mathcal{P} \to (V^+)^\ast$ is not zero.
\label{topological}
\end{lem}

\begin{proof}
For each metric $g$, we define $\omega^i_g \in \Harm^{+_g}(X)\ (1\leq i \leq k)$ by (\ref{harm}).
The image of $g$ by (\ref{section}) belongs to the set
\begin{align}
\Set{f\in\Hom(V^+,\R) = (V^+)^\ast | f(p_{V^+}(\alpha_i)) = c\cdot [\omega^i_g]\ (i=1,\ldots,k)}.
\label{image}
\end{align}
For any positive number $\epsilon > 0$, let take sufficiently large $R_i > 0$ for each $\Si_i$.
Then, by Lemma~\ref{stretch}, (\ref{image}) is contained in the $\epsilon$-neighborhood of $\Hyp_{\Si_i}$ if $g$ is a metric $g_{(\ast,\ldots,\ast,R_{i},\ast,\ldots,\ast)}$ stretched in the neighborhood of $\Si_i$ by the length more than $R_i$.
(Here the neighborhoods of other surfaces can be stretched and not stretched.)
Thus there exists a compact subset $\mathcal{K} \subset \mathcal{P}$ such that $\mathcal{F}(\mathcal{P} \setminus \mathcal{K}) \subset (V^+)^\ast \setminus \{0\}$ and $\mathcal{F} : \mathcal{P} \to (V^+)^\ast$ can be approximated by a continuous map $F : \mathcal{P} \to (V^+)^\ast$ satisfying (\ref{mapstoplane}).
Hence the lemma follows from Condition~\ref{main condition}.
\end{proof}

We now discuss ``wall crossing phenomena" for the moduli space of a family of \SW equations.
Family version of the \SW invariants is studied in Li-Liu~\cite{MR1868921}.
The proofs of the following Lemma~\ref{virtual nbd} and Proposition~\ref{existence} give a proof of non-triviality of a family version of the \SW invariant for a chamber. 

We first consider $S^1$-equivariant and family version of Ruan's virtual neighborhood technique in \cite{MR1635698}.
Let $T$ be a paracompact Hausdorff space and $\Psi : T \to \Met(X)$ be a continuous map.
For each $t \in T$, we set
\begin{align*}
\Harm^{+_t}: = &\Harm^{+_{\Psi(t)}},\\
(\rconfig)_t := &\Ker(d^\ast:L^2_{k}(i\Lambda^1,\Psi(t))\to L^2_{k-1}(i\Lambda^0,\Psi(t))),\\
(\cconfig)_t := &L^2_k(S^+,\Psi(t)),\\
\config_t := &(\rconfig)_t \times (\cconfig)_t,\\
(H_\R)_t := & L^{2}_{k-1}(i\Lambda^{+_t},\Psi(t)),\\
(H_\C)_t : = & L^{2}_{k-1}(S^-,\Psi(t)),\\
H_t : = & (H_\R)_t \times (H_\C)_t.
\end{align*}
Here $L^2_{k}(\cdot,\Psi(t))\ (k \geq 2)$ is the space defined by $L^2_k$-norm with respect to the metric $\Psi(t)$, $\Lambda^p=\Lambda^p T^\ast X$, $\Lambda^+$ is the self-dual part of $\Lambda^2$, and $S^\pm$ are the spinor bundles for the spin c structure corresponding to $c$.
For each $t \in T$, the map corresponding to the \SW equations with respect to the metric $\Psi(t)$ reduces to the $S^1 = U(1)$-equivariant map
\[
s_t : \config_t \to H_t\ ;\ (a,\Phi)\mapsto (F_{A_0+a}^{+_t}-\sigma(\Phi),D_{t,A_0+a}\Phi).
\]
Here $A_0$ is a fixed reference connection on the determinant line bundle for the spin c structure, $F^{+_t}_{A_0 + a}$ is the self-dual part of the curvature of a connection $A_0 + a$ with respect to the metric $\Psi(t)$, and $D_{t, A_0 + a}$ is the Dirac operator with respect to the connection $A_0 + a$ and the metric $\Psi(t)$.
$\sigma(\cdot)$ is the quadratic form
\[
\sigma(\Phi) = \Phi \otimes \Phi^\ast - \frac{|\Phi|^2}{2} \id.
\]
$s_t$ is a non-linear Fredholm map and the index satisfies $\ind{s_t} = d(c) + 1$, where $d(c)$ is the formal dimension of the \SW moduli space for the spin c structure.
We write $\config = \bigsqcup_{t \in T} \config_t,\ \rconfig = \bigsqcup_{t \in T} (\rconfig)_t \cdots,\ s=\{s_t\}_{t \in T}$.

Using $\Psi$, we obtain the vector bundle $\Harm^+ \to T$ by pull-back of the vector bundle
\[
\Harm^+ = \bigsqcup_{g\in \Met(X)} {\Harm^{+_g}(X)} \to \Met(X)
\]
on $\Met(X)$.
Here we fix a homology orientation of $X$.
Then the vector bundle $\Harm^+ \to T$ is oriented.
For each $t \in T$, the affine map
\[
(\rconfig)_t \times \Harm^{+_t} \to (H_\R)_t\ ; \ (a,\alpha) \mapsto F_{A_0+a}^{+_t} + \alpha
\]
is bijective.
(To prove that this map is injective, we have to assume $b_1(X)=0$.)
 We write $((a_0)_t ,(\alpha_0)_t)$ for the unique zero point of this bijective affine map.
This gives a section
\[
f_{\Harm^+} : T \to \Harm^+\ ;\ t \mapsto (\alpha_0)_t
\]
for the vector bundle $\Harm^+ \to T$.
Let $h_t(\cdot)$ be the harmonic part with respect to the metric $\Psi(t)$.
Since the Hodge decomposition $F_{A_0+a} = h_t(F_{A_0+a}) + db_t\ (b_t \in L^2_k(i\Lambda^1, \Psi(t)))$, we have $F_{A_0+a}^{+_t} = -2\pi i \cdot p_{+_t}(c) + d^{+_t}b_t$.
Hence
\[
(\alpha_0)_t - 2\pi i \cdot p_{+_t}(c) = d^{+_t}b_t \in \Harm^{+_t} \cap \im{d^+} = \{0\}.
\]
holds.
Thus we obtain
\begin{align}
f_{\Harm^+}(t) = 2\pi i \cdot p_{+_{g_t}}(c).
\label{eq : two sections}
\end{align}
We can define the relative Euler class
\[
e(\Harm^+ \to T, f_{\Harm^+}) \in H^{b^+}(T, T \setminus f_{\Harm^+}^{-1}(0) ; \Z)
\]
by taking pull-back of the Thom class of the vector bundle $\Harm^+ \to T$ using the section $f_{\Harm^+}$.
Here we consider the following condition for $T$ and $\Psi$.

\begin{condition}
The zero sets $s^{-1}(0)$ and $f_{\Harm^+}^{-1}(0)$ are compact.
In addition, for a subset  $T' \subset T$, $f_{\Harm^+}$ is nowhere vanishing on $T'$, and on $T'$ the parametrized moduli space $s^{-1}(0)|_{T'} = \bigsqcup_{t \in T'} s^{-1}_t(0)$ is empty.
(We also treat the case of $T' = \emptyset$.)
\label{compactness}
\end{condition}

Under the assumption that $T$ and $\Psi$ satisfy Condition~\ref{compactness}, we now construct the  $S^1$-equivariant and family version of Ruan's virtual neighborhood.
Since $s^{-1}(0)$ is compact, there exist the following five data
\begin{enumerate}
\item natural numbers $n=n_1+\cdots+n_k,\ m=m_1+\cdots+m_l\ (n_i,\ m_j\geq0)$,
\item $t_{\R,i},\ t_{\C,j} \in T\ (1\leq i \leq k,\ 1\leq j \leq l)$,
\item real linear maps $\phi_{\R,i}:\R^{n_i}\to H_{t_{\R,i}}\ (1\leq i \leq k)$,
\item complex linear maps $\phi_{\C,j} : \C^{m_j} \to (H_\C)_{t_{\C,j}}\ (1\leq j \leq l)$,
\item partitions of unity $\{\rho_{\R,i}\},\ \{\rho_{\C,j}\}$ subordinated to a finite open covering of $s^{-1}(0)$
\end{enumerate}
such that the differential along the fiber for $(\config \times_{T} \Harm^+) \times \R^n \times \C^m \to T$ of the section
\[
s + \varphi : (\config \times_{T} \Harm^+) \times \R^n \times \C^m \to \left((\config \times_{T} \Harm^+) \times \R^n \times \C^m\right) \times_{T} H
\]
for the Hilbert bundle is surjective at its zero points for each $t \in T$, where $\varphi$ is defined by
\begin{align*}
& \varphi : (\config \times_{T} \Harm^+) \times \R^n \times \C^m \to H\ ;\ \\
(x,\alpha,v,w) & = (x,\alpha,(v_1,\ldots,v_k),(w_1,\ldots,w_l))\\
& \mapsto \alpha + \sum_{i=1}^k\rho_{\R,i}(x)\phi_{\R,i}(v_i) + \sum_{j=1}^l \rho_{\C, j}(x) \phi_{\C,j}(w_j).
\end{align*}
Here we can take $\varphi$ being $S^1$-equivarinat and the partition of unity $\{\rho_{\R,i}\}$ vanishing on the $S^1$-invarinat set $\config^{S^1}$.
Since the condition that a map is surjective is an open condition, there exists a neighborhood $\mathscr{N}$ of $\config \times \{0\}\times \{0\}\times \{0\}$ in $(\config \times_{T} \Harm^+) \times \R^n \times \C^m$ such that the differential of $s+\varphi$ along the fiber is surjective on 
\[
\vnU := (s+\varphi)^{-1}(0) \cap \mathscr{N}.
\]
$\vnU_t = (s_t + \varphi_t)^{-1}(0) \cap \mathscr{N}$ is an $(\ind{s_t} + b^+ + n + 2m)$-dimensional manifold for each $t \in T$.
The $S^1$-invariant set $\vnU^{S^1}$ is the form of
\[
\vnU^{S^1} = \bigsqcup_{t \in T}{\vnU_t}^{S^1}=\bigsqcup_{t \in T}\{((a_0)_t,0,(\alpha_0)_t)\} \times  U \times \{0\} \simeq\bigsqcup_{t \in T}\{((a_0)_t,(\alpha_0)_t)\} \times  U,
\]
where $U$ is a small neighborhood of the origin in $\R^n$.
Take an $S^1$-invariant neighborhood $\mathcal{N}(\vnU^{S^1})$ of $\vnU^{S^1}$.
Write the projection by $\pi : \config \times_T \Harm^+ \times \R^n \times \C^m \to \R^n \times \C^m$.
The section
\[
f : \vnU\setminus \mathcal{N}(\vnU^{S^1}) \to \mathcal{E}
\]
for the vector bundle
\[
\mathcal{E} := \left(\vnU\setminus \mathcal{N}(\vnU^{S^1})\right) \times _T \Harm^+ \times \R^n \times \C^m \to \vnU\setminus \mathcal{N}(\vnU^{S^1})
\]
is obtained by considering for each $t \in T$
\begin{align*}
\left( \vnU\setminus \mathcal{N}(\vnU^{S^1}) \right)_t &\to \Harm^{+_t} \times \R^n \times \C^m\ ;\ x \mapsto ((\alpha_0)_t, \pi(x)).
\end{align*}

For each $t$, we identify the normal bundle $\nu_t \to (\vnU^{S^1})_t$ with a tubular neighborhood of $(\vnU^{S^1})_t$.
Set $\nu : = \bigsqcup_{t \in T}\nu_t$, $S(\nu) :=\bigsqcup_{t\in T}S(\nu_t)$, where $S(\nu_t)$ is the sphere bundle of $\nu_t$.
Then
\begin{align*}
\partial \left( \vnU \setminus \mathcal{N}(\vnU^{S^1}) \right) = S(\nu)
\label{boundart}
\end{align*}
holds.
By the identification
\begin{align*}
\nu \simeq \bigsqcup_{t \in T}\{((a_0)_t, (\alpha_0)_t)\} \times  U \times \C^k \simeq T \times  U \times \C^k,
\end{align*}
we regard $S(\nu) \simeq T \times  U \times S(\C^k)$ and we write $S(\nu)|_{T\times D(\R^n)}$ for
\[
S(\nu)|_{\bigsqcup_{t \in T}\{((a_0)_t, (\alpha_0)_t)\} \times D(\R^n)}.
\]
Here we may assume that the radius of the disk $D(\R^n)$ is sufficiently small, and the neighborhood $U$ of the origin in $\R^n$ satisfies $U \simeq \Int(D(\R^n))$.
Let $f_{\R} : D(\R^n) \to\underline{\R}^n,\ f_{\C} : S(\C^k) \to\underline{\C}^m$ be the sections obtained by the maps
 \begin{align*}
 D(\R^n) \to \R^n\ ;\ v \mapsto v,\ S(\C^k) \to \C^m\ ;\ u \mapsto 0.
 \end{align*}
Since we assumed that $(f_{\Harm^+})^{-1}(0)$ is compact, using a deformation by an $S^1$-equivariant homotopy preserving compact support and the product formula of Euler class, we obtain
\begin{align}
&e_{S^1}\left( \mathcal{E}|_{S(\nu)} \to S(\nu), f \right) \nonumber \\
=& e\left( \Harm^+ \to T ,f_{\Harm^+}\right) \cup e\left( \underline{\R}^n \to D(\R^n) ,f_{\R}\right) \cup e_{S^1}\left( \underline{\C}^m \to S(\C^k),f_{\C}\right)
\label{Euler classes}
\end{align}
under the ismorphism
\begin{align}
&H_{S^1}^{b^+ + n +2m}\left( S(\nu), S(\nu)|_{T' \times D(\R^n) \cup T \times S(\R^n)};\Z \right)\nonumber \\
\simeq & \bigoplus_{p + q + r = b^+ + n +2m} H^p( T, T';\Z) \otimes H^q(D(\R^n), S(\R^n);\Z) \otimes H_{S^1}^r(S(\C^k);\Z).
\label{Kunneth}
\end{align}
Here $e_{S^1}(\cdot)$ is the $S^1$-equivariant Euler class.
Since $c_1^2 > \sign(X)$, by considering the relation between $\dim\vnU_t$ and the formal dimension of the \SW moduli space, we have $m < k$.
Let $\mu \in H^2((\config\setminus \rconfig \times \{0\})/S^1;\Z)$ be the first Chern class of the principal $S^1$-bundle $\config\setminus \rconfig \times \{0\} \to (\config\setminus \rconfig \times \{0\})/S^1$.
We write $\alpha$ for
\[
\alpha := \mu^{2(k-1)-2m} \in H^{2(k-1)-2m}((\config\setminus \rconfig \times \{0\})/S^1;\Z).
\]
Set $\tilde{\config}=\config\times_T \Harm^+\times \R^n \times \C^m$ and let
\[
p : (\tilde{\config} \setminus \tilde{\config}^{S^1})/S^1 \to (\config \setminus \config^{S^1})/S^1=(\config\setminus \rconfig \times \{0\})/S^1
\]
be the projection.

\begin{lem}
Suppose that $T$ and $\Psi$ satisfy Condition~\ref{compactness}, and construct a virtual neighborhood $\vnU$ as above.
Then the following statements hold:
\begin{enumerate}
\item The cohomology class
\begin{align*}
&e_{S^1}\left( \mathcal{E}|_{S(\nu)} \to S(\nu) , f |_{S(\nu)} \right) \cup p^\ast \alpha\\
\in& H_{S^1}^{b^+ + n +2k-2}\left( S(\nu), S(\nu)|_{T' \times D(\R^n) \cup T \times S(\R^n)};\ \Z \right)
\end{align*}
corresponds to 
\begin{align*}
e\left(\mathcal{H}^+ \to T, f_{\mathcal{H}^+}\right) \in H^{b^+}(T, T' ;\ \Z)
\end{align*}
via the isomorphism (\ref{Kunneth}) and 
\[
H^{n}(D(\R^n), S(\R^n);\Z) \simeq \Z,\ H_{S^1}^{2k-2}(S(\C^k);\Z) \simeq \Z.
\]
\item Suppose also that $T$ is a $b^+$-dimensional oriented compact manifold and $T' = \partial T$.
Then for the fundamental class of $T$ (in the case of $T' \neq \emptyset$, fundamental class as a compact manifold with boundary) $[T] \in H_{b^+}(T, T' ;\Z)$, we have
\[
\left<e\left(\Harm^+ \to T , f_{\Harm^+}\right), [T]\right> = 0.
\]
\end{enumerate}
\label{virtual nbd}
\end{lem}

\begin{proof}[Proof of Lemma~\ref{virtual nbd}(1)]
Note that $e\left( \underline{\R}^n \to D(\R^n) ,f_{\R}\right)$ is the generator of
\[
H^{n}(D(\R^n), S(\R^n);\Z) \simeq \Z\]
and
\[
e_{S^1}\left( \underline{\C}^m \to S(\C^k),f_{\C}\right)\cup p^\ast\alpha \in H_{S^1}^{2m + \{2(k-1) - 2m\}}(S(\C^k);\Z)
\]
is the generator of
\[
H_{S^1}^{2m + \{2(k-1) - 2m\}}(S(\C^k);\Z) \simeq H^{2k-2}({\CP}^{k-1};\Z) \simeq \Z.
\]
Hence the claim follows from (\ref{Euler classes}).
\end{proof}

\begin{proof}[Proof of Lemma~\ref{virtual nbd}(2)]
Since
$\partial \left(\vnU \setminus \mathcal{N}(\vnU^{S^1})\right)/S^1
=S(\nu)/S^1 \simeq T \times U \times \CP^{k-1}$ holds,
we have
\begin{align*}
& \left<e\left(\Harm^+ \to T , f_{\Harm^+}\right), [T]\right>\\
=& \left<e_{S^1}\left( \mathcal{E}|_{S(\nu)} \to S(\nu) , f |_{S(\nu)} \right) \cup p^\ast\alpha ,[S(\nu)/S^1]_{BM} \right>\\
=& \left<e_{S^1}\left( \mathcal{E}|_{S(\nu)} \to S(\nu) , f |_{S(\nu)} \right)\cup p^\ast\alpha ,\partial \left[ \left(\vnU \setminus \mathcal{N}(\vnU^{S^1})\right)/S^1 \right]_{BM} \right> = 0.
\end{align*}
Here the subscript $BM$ means Borel-Moore homology.
\end{proof}

\begin{rem}
\begin{enumerate}
\item If $T$ is compact, it automatically follows that $s^{-1}(0)$ is compact by the same argument of usual compactness of the moduli space of the \SW equations.
\item  For Lemma~\ref{virtual nbd} (2), we can replace the above argument by the alternative one with $\Z/2$-coefficient in the case when $T$ is a manifold which is not orientable.
\end{enumerate}
\end{rem}

Before stating the next proposition, note that
\begin{align}
\Harm^{+_g}(X) \to \Hom(V^+,\R) = (V^+)^\ast\ ;\ \omega\mapsto(v\mapsto [\omega] \cdot [\varphi^{-1}_{g}(v)])
\label{trivialization}
\end{align}
is a linear isomorphism since the intersection form restricted to $\Harm^{+_g}$ is positive definite (in paticular non-degenerate).
By using this isomorphism, we obtain a trivialization of the vector bundle $\Harm^+ \to\Met(X)$.
By (\ref{eq : two sections}), via this trivialization, the section $f_{\Harm^+} : T \to \Harm^+$ corresponds to the section
\[
T \to T \times (V^+)^\ast\ ;\ t \mapsto (t, ( v \mapsto 2 \pi i [p_{+_t}(c)] \cdot [\varphi^{-1}_{g_t}(v)]))
\]
of the trivial bundle $T \times (V^+)^\ast \to T$.
This is the restriction of the section $\Met(X) \to \Met(X) \times (V^+)^\ast$ obtained from (\ref{section}) up to constant.
For $\vec{R} = (R_1,\ldots,R_k) \in [0,\infty)^k$, we write $\mathcal{P}(\vec{R})$ for
\[
\mathcal{P}(\vec{R}) := \bigcap_{i=1}^k \Set{\varphi\in \mathcal{P} | \left<\varphi,v_i\right>\leq R_i}.
\]
Note that there exists a continuous map from the $b^+$-dimensional disk $D^{b^+}$ to $\mathcal{P}(\vec{R})$ which induces an isomorphism $H^{b^+}(\mathcal{P}(\vec{R}), \partial(\mathcal{P}(\vec{R}));\Z) \simeq H^{b^+}(D^{b^+}, S^{b^+-1};\Z)$ if $\Si_1, \ldots, \Si_k$ and $c$ satisfy Condition~\ref{main condition} and all of $R_1, \ldots, R_k$ are positive.

\begin{prop}
Assume that $b_1(X) = 0,\ c^2 > \sign(X)$, and in addition, $\Si_1, \ldots, \Si_k$ and $c$ satisfy Condition~\ref{main condition}.
Then for sufficiently large $R_0 > 0$, there exits a metric $g \in \partial(\iota(\mathcal{P}(\vec{R}_0)))$ such that there exists a solution of the \SW equations with respect to the metric $g$ (and the spin c structure corresponding to $c$), where $\vec{R}_0 := (R_0, \ldots, R_0)$.
\label{existence}
\end{prop}

\begin{proof}
In the setting of Lemma~\ref{virtual nbd}, let $T = D^{b^+}$ and $\Psi$ be a continuous map $D^{b^+} \to \iota(\mathcal{P}(\vec{R}_0))$ which induces an isomorphism between the relative cohomology groups.
By Lemma~\ref{topological}, $f_{\Harm^+}$ is nowhere vanishing on $S^{b^+-1}$ and we have
\[
\left < e\left(\Harm^+ \to D^{b^+} ,f_{\Harm^+}\right), [D^{b^+}] \right> \neq 0,
\]
where $[D^{b^+}]$ is the fundamental class $[D^{b^+}]\in H_{b^+}(D^{b^+} , S^{b^+-1}; \Z)$.
Hence the parametrized moduli space on $S^{b^+-1}$ is not empty by Lemma~\ref{virtual nbd} (2).
\end{proof}

\begin{rem}
Nakamura~\cite{MR2015245} used a family version of the \SW equations and information coming from reducible solutions to study diffeomorphisms on $4$-manifold.
Kronheimer-Mrowka-Ozsv\'{a}th-Szab\'{o}~\cite{MR2299739} used a 2-parameter family of \SW equations to consider an exact triangle for monopole Floer homology.
Ruberman~\cite{MR1734421}, \cite{MR1734421} and \cite{MR1874146} used 1-parameter version Donaldson/\SW invariants and the wall crossing phenomena to study diffeomorphisms and metrics with positive scalar curvature on $4$-manifold.
\end{rem}

The following proposition will be proved in the next section.

\begin{prop}
Let $\Si_1,\ldots,\Si_l$ be the same one in Lemma~\ref{stretch}.
Suppose that there exists a constant $C_0 > 0$ depending on only the initial metric $g_0$ and $c$,  there exist non-negative numbers $R_1, \ldots, R_l\geq0$ satisfying $R_1+\cdots+R_l>C_0/2$ and there exists a solution of the \SW equations with respect to a metric $g_{(R_1,\ldots,R_l)}$ stretched in the neighborhood of each $\Si_i$ by the length $R_i$.
Then
\begin{align}
\chi^-(\Si_i) \geq |c \cdot \alpha_i|
\label{pre-adj. for several}
\end{align}
holds for at least one $i \in \{1,\ldots,l\}$.
(The constant $C_0$ will be given concretely in (\ref{pre-adjunction}) in Lemma~\ref{pre-adjunction lemma}.)
\label{analysis}
\end{prop}

Then we can complete the proof of Theorem~\ref{general}.
Take sufficiently large $R_0 > 0$.
By Proposition~\ref{existence}, there exists a solution of the \SW equations with respect to a metric on $\partial(\iota(\mathcal{P}(\vec{R_0})))$.
Therefore there exist surfaces $\Si_{i_1}, \ldots, \Si_{i_l}\ (1 \leq l \leq b^+,\ \{i_1, \ldots, i_l\} \in\mathcal{S})$ and there exist numbers $R_{i_1}, \ldots, R_{i_l} \geq 0$ such that $R_{i_1} +\cdots + R_{i_l} > R_0/2$ and there exists a solution of the \SW equations with respect to a metric $g_{(R_{i_1}, \ldots, R_{i_l})}$ stretched in the neighborhood of $\Si_{i_1}, \ldots, \Si_{i_l}$ by the length $R_{i_1}, \ldots, R_{i_l}$ respectively.
Hence by Proposition~\ref{analysis}, the inequality holds for at least one of $\Si_{i_1}, \ldots, \Si_{i_l}$.
Since we assume that $c \cdot \alpha_i \neq 0$, if we have $\chi^-(\Si_i) \geq | c \cdot \alpha_i |$,then $\chi^-(\Si_i) = -\chi(\Si_i)$ holds.
This proves Theorem~\ref{general}.

\section{Analytical arguments}

In this section, we prove Lemma~\ref{stretch} and Proposition~\ref{analysis} in the previous section.
Let $X$ be an oriented closed smooth $4$-manifold and $\Si\subset X$ be an embedded surface with zero self-intersection number and set $\alpha : = [\Si]$.
We take an initial metric $g_0$ such that a neighborhood of the sphere bundle of the normal bundle of $\Si$ forms $[0,1] \times S^1 \times \Si =: [0,1] \times Y$ with the product metric.
We consider a metric $g_R$ stretched in the neighborhood of $\Si$ by the length $R>0$.
(In the following argument, it is not necessary that $g_R$ coincides with $g_0$ on the complement of the neighborhoods of $\Si$.)
Let $c \in H^2(X;\Z)$ be a characteristic.

The following Lemma~\ref{cylinder} is shown by Cauchy-Schwarz inequality and Fubini's theorem.

\begin{lem}(Kronheimer-Mrowka~\cite{MR1492131},\ Kronheimer~\cite{MR1664695})
For any closed 2-form $\omega$ on $X$,
\begin{align}
|\alpha \cdot [\omega]|^2=\left|\int_{\Si} \omega \right|^2\leq\frac{\|\omega\|^2_{L^2_{g_R}([0,R]\times Y)}\cdot\Area({\Si})}{\Area([0,R]\times S^1)}
\label{key estimate}
\end{align}
holds, where $\Area(\cdot)$ means the area with respect to $g_0$ and $L^2_g$ means $L^2$-norm with respect to the metric $g$.
\label{cylinder}
\end{lem}

\begin{lem}
For a metric $g$, we define $\omega_g := \varphi_g^{-1}(p_{V^+}(\alpha)) \in \Harm^{+_g}(X)$, then
\begin{align*}
|\alpha - [\omega_{g_R}]|^4 \leq \frac{C\cdot\Area({\Si})}{\Area([0,R]\times S^1)}
\end{align*}
holds for a positive number $R > 0$.
Therefore we have
\[
\lim_{R \to \infty}[\omega_{g_{R}}] = \alpha.
\]
Here the constant $C > 0$ depends only on the homology class $\alpha \in H_2(X ; \R)$ and the fixed decomposition $H^2(X ; \R) = V^+ \oplus  V^-$.
(In particular, $C$ is independent of the positive number $R$ and the choice of the metric $g_R$.)
\label{close}
\end{lem}

\begin{proof}
Note that for any metric $g$, $\|\omega_g\|_{L^2_g(X)} \leq |p_{V^+}(\alpha)|^2 =: C_0$ holds.
By Lemma~\ref{cylinder}, we have
\begin{align}
|\alpha \cdot [\omega_{g_R}]|^2 = \left|\int_{{\Si}}\omega_{g_R}\right|^2 \leq \frac{C_0\cdot\Area({\Si})}{\Area([0,R]\times S^1)}.
\label{key-estimate}
\end{align}
Consider the decomposition $\alpha = \beta + \gamma,\ [\omega_{g_R}] = \beta + \delta\ (\beta\in V^+,\ \gamma,\delta\in V^-)$.
Since we have $-2(\beta^2 + \gamma\delta) = (\gamma - \delta)^2 - (\beta + \delta)^2$ and the both terms of the right-hand side are negative, we obtain $|2(\beta^2+\gamma\delta)| \geq |(\gamma - \delta)^2| = |\gamma - \delta|^2$, i.e., $2|\alpha\cdot[\omega_{g_R}]| \geq |\alpha - [\omega_{g_R}]|^2$.
By putting together (\ref{key-estimate}), we have
\[
|\alpha - [\omega_{g_R}]|^2
 \leq 2 \left(\frac{C_0 \cdot \Area({\Si})}{\Area([0,R] \times S^1)}\right)^{1/2}.
\]
\end{proof}

\begin{proof}[Proof of Lemma~\ref{stretch}]
It is clear since in Lemma~\ref{close} the constant $C$ does not depend on the choice of the metric $g_R$.
\end{proof}

The following Lemma~\ref{L^2-a priori} is shown by using the Weitzenb\"{o}ck formula.

\begin{lem}(Kronheimer-Mrowka~\cite{MR1492131},\ Kronheimer~\cite{MR1664695})
Let $s_g : X \to \R$ be the scalar curvature for a metric $g$.
We define $\kappa_g : X \to \R$ by
\[
\kappa_g := \max\{ -s_g, 0 \}.
\]
If there exists a solution of the \SW equations with respect to the metric $g$ and the spin c structure corresponding to $c$, then
\begin{align}
\|h_g(c)\|_{L^2_g(X)}^2
\leq \frac{\| \kappa_g \|_{L^2_g(X)}^2}{(4\pi)^2} - \int_X c^2.
\label{char-estimate}
\end{align}
\label{L^2-a priori}
\end{lem}

\begin{lem}(cf. Kronheimer-Mrowka~\cite{MR1492131}, Kronheimer~\cite{MR1664695})
Normalize the initial metric $g_0$ to be on $\Si$ of constant scalar curvature and of unit area, and in addition, of unit length for $S^1$ in $[0,1]\times S^1\times\Si$.
If there exists a solution of the \SW equations with respect to the metric $g_R$ and $c$, then
\begin{align}
\chi^-(\Si)^2 + \frac{C_0}{{R}}
\geq | c \cdot \alpha |^2,\ {\rm where\ } C_0 = \frac{\| \kappa_{g_0} \|_{L^2_{g_0}(X)}^2}{(4\pi)^2} - \int_X c^2.
\label{pre-adjunction}
\end{align}
(Note that $C_0$ depends on only $g_0$ and it can be positive, negative and zero.)

Thus if there exists a positive number $R$ such that $R > |C_0|/2$ and there exists a solution of the \SW equations with respect to the metric $g_R$, we have
\[
\chi^-(\Si)
\geq | c \cdot \alpha |.
\]
\label{pre-adjunction lemma}
\end{lem}

\begin{proof}
It follows immediately from Lemma~\ref{cylinder} and Lemma~\ref{L^2-a priori}.
\end{proof}

\begin{proof}[Proof of Proposition~\ref{analysis}]
For simplicity, we give the proof only in the case of $l = 2$.
Let $Y_i := S^1 \times \Si_i$.
Assume that the inequality (\ref{pre-adj. for several}) does not hold for both $i = 1, 2$.
Then, since $\chi^-(\Si_i)^2 + 1 \leq | c \cdot \alpha_i |^2$ holds, taking weighted average, we have
\begin{align}
{R}_1 \chi^-(\Si_1)^2 + {R}_2 \chi^-(\Si_2)^2 + 2(R_1 + R_2)
\leq {R}_1 |c \cdot \alpha_1 |^2 + {R}_2 | c \cdot \alpha_2|^2.
\label{estimate0}
\end{align}

On the other hand,
\begin{align}
\| h_{g_{(R_1, R_2)}}(c) \|_{L_{g_{(R_1, R_2)}}^2(X)}^2
\leq {R}_1 \chi^-(\Si_1)^2 + {R}_2 \chi^-(\Si_2)^2 + \frac{\| \kappa_{g_0} \|^2_{L^2_{g_0}(X)}}{(4\pi)^2} - \int _X c^2,
\label{estimate2.5}
\end{align}
holds by Lemma~\ref{L^2-a priori}.
We also have
\begin{align}
| c \cdot \alpha_i |^2
\leq \frac{\| h_{g_{(R_1, R_2)}}(c) \|^2_{L^2_{g_{(R_1, R_2)}}([-R_i, R_i] \times Y_i)}}{{R}_i}\ (i = 1, 2)
\label{estimate3}
\end{align}
by Lemma~\ref{cylinder}.
Taking weighted average for (\ref{estimate3}), we obtain
\begin{align}
&{R}_1 | c \cdot \alpha_1 |^2 + {R}_2 | c \cdot \alpha_2|^2 \nonumber \\
\leq & \| h_{g_{(R_1, R_2)}}(c) \|^2_{L^2_{g_{(R_1, R_2)}}( [-R_1, R_1] \times Y_1 )} + \| h_{g_{(R_1, R_2)}}(c) \|^2_{L^2_{g_{(R_1, R_2)}}([-R_2, R_2] \times Y_2)} \nonumber \\
\leq & \| h_{g_{(R_1, R_2)}}(c) \|^2_{L^2_{g_{(R_1, R_2)}}(X)}
\label{estimate5}.
\end{align}
By putting together (\ref{estimate0}), (\ref{estimate2.5}) and (\ref{estimate5}),
\begin{align*}
&{R}_1 \chi^-(\Si_1)^2 + {R}_2 \chi^-(\Si_2)^2 + 2 (R_1 + R_2)\nonumber\\
\leq&{R}_1 \chi^-(\Si_1)^2 + {R}_2 \chi^-(\Si_2)^2
+ \frac{\| \kappa_{g_0} \|^2_{L^2_{g_0}(X)}}{(4\pi)^2} - \int _X c^2
\end{align*}
holds.
Therefore we obtain
\begin{align*}
2 (R_1 + R_2)
\leq \frac{\| \kappa_{g_0} \|^2_{L^2_{g_0}(X)}}{(4\pi)^2} - \int _X c^2.
\end{align*}
This contradicts our assumption.
\hspace{\fill}
\end{proof}

\begin{rem}
We can also prove Proposition~\ref{analysis} by the original argument in Kronheimer-Mrowka~\cite{MR1306022}; the method using Chern-Simons-Dirac functional (\cite{MR1306022},\ Proposition 8, Lemma 9).
Using the argument, one can show easily the following proposition.
Let $l > 0$ and $\Si_1,\ldots,\Si_l$ and $g_{(R_1, \ldots, R_l)}$ be the same one in Proposition~\ref{analysis}.

\begin{prop}
Assume that for any sufficiently large $R > 0$ there exist positive numbers $R_1, \ldots, R_l > 0$ such that $\min\{R_1, \ldots, R_l\} > R$ and there exists a solution of the \SW equations with respect to the metric $g_{(R_1, \ldots, R_l)}$. 
Then for each $i \in \{ 1, \ldots, l \}$ there exists a solution on $\R \times Y_i$ which is translation-invariant in temporal gauge.
\label{CSD}
\end{prop}

From Proposition~\ref{CSD} together with \cite{MR1306022} Lemma 9, we obtain a result which is stronger than Proposition~\ref{analysis}:

If the assumption in Proposition~\ref{CSD} is satisfied, then for any $i \in \{ 1, \ldots, l \}$, 
\begin{align*}
\chi^-(\Si_i) \geq | c \cdot \alpha_i |.
\end{align*}

\end{rem}

\end{document}